\documentclass[12pt]{amsart}

\usepackage{amsmath,amssymb,latexsym} 
\usepackage{fullpage}
\usepackage{etoolbox}
\usepackage{enumitem}
\usepackage{color}
\usepackage[colorinlistoftodos,bordercolor=orange,backgroundcolor=orange!20,linecolor=orange,textsize=scriptsize]{todonotes}
\usepackage{enumitem}

\newtheorem{theorem}{Theorem} 

\theoremstyle{definition}

\theoremstyle{remark}

\newtheorem*{proof-claim}{Proof}

\newenvironment{changemargin}[2]{\begin{list}{}{%
\setlength{\topsep}{0pt}%
\setlength{\leftmargin}{0pt}%
\setlength{\rightmargin}{0pt}%
\setlength{\listparindent}{\parindent}%
\setlength{\itemindent}{\parindent}%
\setlength{\parsep}{0pt plus 1pt}%
\addtolength{\leftmargin}{#1}%
\addtolength{\rightmargin}{#2}%
}\item }{\end{list}}

\AtBeginEnvironment{proof-claim}{\vspace{-0.5cm}\footnotesize\begin{changemargin}{0.5cm}{0.5cm}}
\AtEndEnvironment{proof-claim}{\end{changemargin}}

\def\L{\mathsf{L}}

\def\K{\mathsf{K}}

\begin{document} 

\title{Rainbow simplices in triangulations of manifolds}  
\author{ Luis Montejano}

\maketitle 

\begin{abstract} 
Given a coloration of the vertices of a triangulation of a manifold, we give homological conditions on the chromatic complexes under which it is possible to obtain a rainbow simplex
\end{abstract}
\section{Introduction and preliminaries}
Consider a simplicial complex $\K$ which is a triangulation of a $n$-dimensional manifold and  whose vertices are partitioned into $n+1$ subsets $V_0,\ldots,V_n$.  Following the spirit of the Sperner lemma, the purpose of this paper is to obtain conditions  that allow us to ensure the existence of a rainbow simplex, that is, an $n$-simplex of $\K$ with exactly one vertex in each $V_i$.  In particular, we  will be interested in give homological conditions on the chromatic complexes $\K_{\{i\}}$, where 
we denote by $\K_{\{i\}}$ the subcomplex of $\K$ generated by the vertices of $V_i$.

\bigskip

During this paper, we use reduced homology with coefficients in a arbitrary field and, if no confusion arise, we shall not distinguish between a simplicial complex and its topological realization. For instance, if $\L$ is a subcomplex of the simplicial complex $\K$,  in this paper we shall denote by $\K\setminus \L$ the space $|\K|\setminus |\L|$.

\bigskip

Meshulam's lemma~\cite[Proposition 1.6]{Mh2} and~\cite[Theorem 1.5]{Mh1}  is a Sperner-lemma type result, dealing with coloured simplicial complexes and rainbow simplices, and in which the classical boundary condition of the Sperner lemma is replaced by an acyclicity condition. It is an important result from topological combinatorics with several applications in combinatorics, such as the generalization of Edmonds' intersection theorem by Aharoni and Berger~\cite{AB} and many other results in which obtaining a system of distinct representatives is relevant, like for example, the Hall's theorem for hypergraphs \cite{AH}. 
Following this spirit, Meunier and Montejano  \cite{MM} generalized Meshulam's lemma obtaining the following result which is the main tool in this paper to obtain rainbow simplices.

\bigskip
  
Consider a simplicial complex $\K$ whose vertex set is partitioned into $V_0,\ldots,V_n$. 
For $S\subseteq \{0,1,\dots n\}$, we denote by $\K_S$ the subcomplex of $\K$ induced by the vertices in 
$\bigcup_{i\in S}V_i$.  
Suppose that  for every nonempty $S\subseteq \{0,1,\dots n\}$,
\begin{equation}\label{eq1}
\widetilde{H}_{|S|-2}(\K_S)=0.
\end{equation} 
 Then  there exists a rainbow simplex $\sigma$ in $\K$. See \cite[Theorem 4]{M} and \cite{MM} for a proof.

\bigskip

We summarize below what we need about PL topology in this paper. See, for example, the book of Rourke and Sanderson \cite{RS}.

\bigskip
Let $\{U_1, U_2\}$ be a partition of the vertices $V(\K)$ of the simplicial complex $\K$ and let $<U_1>$ and $<U_2>$ be the subcomplexes of $\K$ induced by $U_1$ and $U_2$, respectively. 
Let $N(<U_i>,\K')$ be the derived neighborhoods of $U_i$ in $\K$ as subcomplex of the first barycentric subdivision $\K'$, $i=1,2$. Hence:
\begin{itemize}
\item $N(<U_i>,\K')$ is a strong deformation retraction of $<U_i>$, $i=1,2$ and 
\item $\K \setminus <U_2>$ is a strong deformation retraction of $<U_1>$. 
\end{itemize}

\section{Rainbow simplices in triangulations of $2$ and $3$-manifolds}

Our first result deals with $3$-colorations in triangulation of surfaces
\bigskip

\begin{theorem} \label{thmsurface}
Consider a simplicial complex $\K$ which is a triangulation of a $2$-dimensional manifold and  whose vertices are partitioned into $3$ nonempty subsets $V_0,V_1,V_2$. If for every $i=0,1,2$,
$$\widetilde H_1(\K_{\{i\}}, \K_{\{i\}}\cap\partial\K)=0,$$
 then $\K$ admits a rainbow triangle. 
\end{theorem}

\begin{proof} The condition $\widetilde H_1(\K_{\{i\}}, \K_{\{i\}}\cap\partial\K)=0$ implies that every component $\L$ of $\K_{\{i\}}$ is contractible and the intersection of $\L$ with the boundary of $\K$ is either empty or contractible.
Of course, we may assume that $\K$ is connected. Let us assume first that $\K$ is the triangulation of a simply connected surface without boundary. In order to find a rainbow triangle in $\K$, by (\ref{eq1}), it will be enough to prove that
$\K_S$ is  connected,  
for every subset $S\subset\{0,1,2\}$ of size two.  Asume $S=\{1,2\}$ and  let $p$ and $q$ be two points in $\K_{\{1,2\}}$. By hypothesis, $N(\K_{\{0\}},\K')$ is a countable collection $\{D_i\}$ of pairwise disjoint topological disk embedded in the surface $\K$. 
Let $f_i:\mathbb{B}^2\to D_i$ be  homeomorphisms and let 
$\{x_i=f_i(0)\}$ be the collection of centers of all these disks. Since $\K$ is connected, there is an arc $\Gamma$ joining $p$ and $q$ in $\K$. Furthermore, by transversality, we may assume without loss of generality, that this arc $\Gamma$ does not intersect the collection of centers
$\{x_i\}$.  Moreover, use the radial structure of the disks $\{D_i\}$, giving by the homeomorphisms $\{f_i\}$, to push the arc $\Gamma$ outside $\K_{\{0\}}$.  Since $\K_{\{1,2\}}$ is a strong deformation retract of $\K\setminus \K_{\{0\}}$, we can deform de arc $\Gamma$ to an arc from $p$ to $q$ inside $\K_{\{1,2\}}$, those proving the connectivity of $\K_{\{1,2\}}$ as we wished.

Suppose now  that $\K$  is the triangulation of a $2$-dimensional non simply connected manifold. Taking the universal cover of this surface we obtain a simply connected simplicial complex $\widetilde \K$ that inherits from $\K$ a $3$-coloration $\widetilde V_0, \widetilde V_1, \widetilde V_2$ on its vertices. That is, there is a simplicial map $\pi:\widetilde\K\to \K$ which is a universal cover, where $\widetilde V_i=\pi^{-1}(V_i)$, $i=0,1,2$. Note that if $\L$ is contractible, hence 
$\pi^{-1}(\L)$ is a countable union of pairwise disjoint contractible subcomplexes of $\widetilde\K$. Therefore, 
the fact that  every component $\L$ of $\K_{\{i\}}$  is contractible implies that every component of $\widetilde \K_{\{i\}}$  is contractible.  Consequently, by the first part of the proof, there is a rainbow simplex in $\widetilde \K$ and since $\pi$ sends colorful simplices of $\widetilde \K$ into colorful simplices of $\K$, we obtain our desired rainbow triangle.

The proof of the theorem for triangulations of surfaces with boundary is completely similar, except that if for a component $\L$ of $\K_{\{i\}}$ such that $\L$ and $\L\cap\partial \K$ are nonempty contractible spaces, then we use a homeomorphism 
$$f_i:\big(\mathbb{B}^2\cap\{(x,y) \in \mathbb{R}^2 \mid y\geq 0\}, [-1,1]\times\{0\}\big)\to \big(N(\L, \K'), N(\L\cap \partial \K, \partial K')\big)$$
in such a way that the center $x_i=f_i(0)$ lies in $\partial\K$ the boundary of $\K$.
\end{proof}
\bigskip
\bigskip

 For triangulations of $3$-dimensional manifolds, we have the following theorem.
\bigskip

\begin{theorem} \label{thmthree}
Consider a simplicial complex $\K$ which is a triangulation of a $3$-dimensional  manifold   whose vertices are partitioned into $4$ subsets $V_0,V_1,V_2,V_3$. Suppose that 
\begin{enumerate}
\item $\widetilde H_2(\K)=0$,
\item for $i=0,\dots 3$,  $\K_{\{i\}}$ is contractible and the intersection of  $\K_{\{i\}}$ with the boundary of $\K$ is either empty or contractible, 
\item for every pair of integers $0\leq i <j\leq 3$,  there is a $1$-dimensional simplex with one vertex in $V_i$ and the other in $V_j$. 
\end{enumerate}
then $\K$ admits a rainbow tetrahedron. 
\end{theorem}

\begin{proof} As in the proof of Theorem \ref{thmsurface}, we may assume without loss of generality that  $\K$ is a triangulation of a connected, simply connected $3$-dimensional  manifold without boundary.  Since $\widetilde H_2(\K)=0$, in order to get a rainbow simplex of $\K$ it is enough to prove that 
\begin{itemize} 
\item $\widetilde{H}_{1}(\K_S)=0$, for every  $S\subset \{0,1,2,3\}$ of size $3$, and 
\item  $\widetilde{H}_{0}(\K_S)=0$, for every  $S\subset \{0,1,2,3\}$ of size $2$.
\end{itemize}

Indeed, we shall prove that for every  $S\subset \{0,1,2,3\}$ of size $3$, $\K_S$ is simply connected and  for every  $S\subset \{0,1,2,3\}$ of size $2$, $\K_S$ is  connected. Assume that $S=\{1,2,3\}$.
Let  $\alpha:\mathbb{S}^1\to \K_{\{1,2,3\}}$ be a continuous map. Since $\K$ is simply connected, there is a map $\beta:\mathbb{B}^2\to  \K$ extending $\alpha$. Since $\K_{\{0\}}$ is contractible, hence its derived neighborhood $N(\K_{\{0\}},\K')$ is a 3-dimensional ball. 
Asume that for a parametrization $f:\mathbb{B}^3 \to N(\K_{\{0\}},\K')$, the point $x=f(0)$ is its center. As in the proof of Theorem \ref{thmsurface}, by transversality, we may assume without loss of generality, that the point $x$ does not lie in $\beta(\mathbb{B}^2)$.  Moreover, use the radial structure of the $3$-ball  $N(\K_{\{0\}},\K')$, given by the homeomorphism $f$,  to push $\beta(\mathbb{B}^2)$ outside $\K_{\{0\}}$.  Since $\K_{\{1,2,3\}}$ is a strong deformation retract of $\K\setminus \K_{\{0\}}$,  the map $\beta:\mathbb{B}^2\to  \K$  is homotopic to a map whose image lies inside $\K_{\{1,2,3\}}$ and of course extends the map $\alpha$. This proves that $\K_{\{1,2,3\}}$ is simply connected.  Finally,  given two different integers $0\leq i <j\leq m$, the connectivity of $\K_{\{i,j\}}$ follows from the connectivity of $\K_{\{i\}}$ and 
$\K_{\{j\}}$ plus the existence of a $1$-dimensional simplex with one vertex in $V_i$ and the other in $V_j$.  This completes the proof of this theorem.

\end{proof}

\bigskip

\section {Rainbow simplices in triangulations of $n$-dimensional manifolds}

\begin{theorem} \label{thm4surface}
Consider a simplicial complex $\K$ which is a triangulation of a $4$-dimensional closed  manifold and  whose vertices are partitioned into $5$ subsets $V_0, \dots,V_4$.  
\begin{enumerate}
\item $\widetilde H_2(\K)=\widetilde H_3(\K)=0$,
\item for every pair of integers $0\leq i <j\leq 4$   there is a $1$-dimensional simplex with one vertex in $V_i$ and the other in $V_j$,
\item for every $i=0,\dots,4$, the subcomplex $K_{\{i\}}$ is contractible and has as a regular neighborhood a $4-ball$,
\item for every $S\subset\{1,\dots,5\}$ of size $2$,  $K_S$ have as a regular neighborhood  a handle body 
\end{enumerate}

then $\K$ admits a rainbow $4$-simplex. 

\end{theorem}

\begin{proof}
As in the proof of Theorem \ref{thmsurface}, we may assume without loss of generality that  $\K$ is a triangulation of a connected, simply connected $4$-dimensional  closed manifold. 
The strategy is to get a rainbow simplex by proving (\ref{eq1}).
Since $\widetilde H_3(\K)=0$, we have that $\widetilde{H}_{3}(\K_S)=0$, for every  $S\subset \{0,\dots,4\}$ of size $5$. Let us prove that $\widetilde{H}_{2}(\K_S)=0$, for every  $S\subset \{0,\dots,4\}$ of size $4$.  
Asume $S=\{1,2,3,4\}$ and  let $\sigma^2$ be $2$-cycle of $\widetilde{H}_{2}(\K_S)$. Since $\widetilde{H}_{2}(\K)=0$, there is a chain complex $\sigma^3$ in $\K$ whose boundary is $\sigma^2$.  
Since $\K_{\{0\}}$ is contractible, its derived neighborhood  $N(\K_{\{0\}},\K')$ is a 3-dimensional ball, let $f:\mathbb{B}^4 \to N(\K_{\{0\}},\K')$ be a homeomorphism and denote by $x=f(0)$ its center. As in the proof of Theorem \ref{thmsurface}, by transversality, we may assume without loss of generality, that the point $x$ does not lie in $\sigma^3$. 
Moreover, use the radial structure of the $4$-ball  $N(\K_{\{0\}},\K')$, given by the homeomorphism $f$,  to push $\sigma^3$ outside $\K_{\{0\}}$. Since $\K_{S}$ is a strong deformation retract of $\K\setminus \K_{\{0\}}$, $\sigma^3$ is homotopic to a 
chain of complex $\K_{S}$ whose boundary is $\sigma^2$.  Therefore $\widetilde{H}_{2}(\K_S)=0$, for every  $S\subset \{0,\dots,4\}$ of size $4$.  
Let us prove now that $\widetilde{H}_{1}(\K_S)=0$, for every  $S\subset \{0,\dots,4\}$ of size $3$.  
Asume $S=\{2,3,4\}$ and  let $\delta^1$ be $1$-cycle of $\widetilde{H}_{1}(\K_S)$. Since $\widetilde{H}_{1}(\K)=0$, there is a chain complex $\delta^2$ in $\K$ whose boundary is $\delta^1$.  
Moreover, since the derived neighborhood $N(\K_{\{0,1\}},\K')$ is a handle body, there is a $1$-dimensional subpolyhedron $L^1$ with the property $L^1$ is a strong deformation retract of $N(\K_{\{0,1\}},\K')$. By transversality, we may assume that 
$\delta^2$ does not intersect $L^1$. Moreover, we can use the radial structure of $N(\K_{\{0,1\}},\K')\setminus L^1$ to push $\delta^2$ 
outside $\K_{\{0,1\}}$. Since $\K_{S}$ is a strong deformation retract of $\K\setminus \K_{\{0,1\}}$, $\delta^2$ is homotopic to a 
chain complex of $\K_{S}$ whose boundary is $\delta^1$.  Therefore $\widetilde{H}_{1}(\K_S)=0$, for every  $S\subset \{0,\dots,4\}$ of size $3$.  
Finally,  given two different integers $0\leq i <j\leq m$, the connectivity of $\K_{\{i,j\}}$ follows from the connectivity of $\K_{\{i\}}$ and 
$\K_{\{j\}}$ plus the existence of a $1$-dimensional simplex with one vertex in $V_i$ and the other in $V_j$.  This completes the proof of  Theorem \ref{thm4surface}.

\end{proof}

The same ideas in our previos theorems can be applied to obtain the following theorem.

\bigskip

\begin{theorem} \label{thmnsurface}
Consider a simplicial complex $\K$ which is a triangulation of a $n$-dimensional closed manifold and  whose vertices are partitioned into $n+1$ subsets $V_0, \dots,V_n$.  If 
\begin{enumerate}
\item $\widetilde H_2(\K)=\dots =\widetilde H_{n-1}(\K)=0$,
\item for every $S\subset\{0,\dots,n\}$ of size $i+1$,  there is a $i$-dimensional complex $L^i$ such that $L^i$ strong deformation retracts to $N(\K_S, \K')$,  $0\le i\le n-2$,
\end{enumerate}
then $\K$ admits a rainbow $n$-simplex. 

\end{theorem}

Finally, we shall use Alexander duality to prove the following theorem for colored  triangulations of spheres

\bigskip

\begin{theorem} \label{thmsphere}
Consider a simplicial complex $\K$ which is a triangulation of the $n$-dimensional sphere and  whose vertices are partitioned into $n+1$ subsets $V_0, \dots,V_n$.  If for every $S\subset\{0,\dots,n\}$ and 
$1\le |S|\le n-1$,
$$ \widetilde H_{|S|}(K_S)=0,$$
then $\K$ admits a rainbow $n$-simplex. 
\end{theorem}

\begin{proof}
Let $S\subset [m]=\{0,1,\dots,m\}$. Then $\K-\K_S$ has the homotopy type of $K_{[m]\setminus S}$. By Alexander duality, $\widetilde H_{|S|-2}(K_S)=\widetilde H_{m+1-|S|}(K_{[m]\setminus S})=0$. 
Consequently, by (\ref{eq1}), $\K$ admits a rainbow $n$-simplex.

\end{proof}
\bigskip

\section{Acknowledgements}
The  author wish to acknowledge  support  form CONACyT under 
project 166306 and  support from PAPIIT-UNAM under project IN112614.

\end{document}